\documentclass[12pt]{amsart}
\usepackage{amssymb,amsthm,mathrsfs,enumitem}

\newcommand{\height}{\operatorname{ht}}
\newcommand{\ara}{\operatorname{ara}}
\newcommand{\Spec}{\operatorname{Spec}}
\renewcommand{\phi}{\varphi}

\newcommand{\Ima}{\operatorname{Im}}
\newcommand{\Max}{\operatorname{Max}}
\newcommand{\Min}{\operatorname{Min}}

\newcommand{\Gold}{\operatorname{Gold}}

\newtheorem{proposition}{Proposition}[section]

\newtheorem{corollary}[proposition]{Corollary}
\newtheorem{theorem}[proposition]{Theorem}

\theoremstyle{definition}

\newtheorem{example}[proposition]{Example}
\newtheorem{remark}[proposition]{Remark}

\begin{document}

\setlist[enumerate]{label=({\bfseries\roman*}), font=\normalfont}

\title{Avoidance and absorbance}

\author[A. Tarizadeh and J. Chen]{Abolfazl Tarizadeh and Justin Chen}
\address{Department of Mathematics, Faculty of Basic Sciences, University of Maragheh 
P. O. Box 55136-553, Maragheh, Iran.}
\email{ebulfez1978@gmail.com}
\address{School of Mathematics, Georgia Institute of Technology, Atlanta, Georgia, 30332 U.S.A.}
\email{justin.chen@math.gatech.edu}

\date{}

\begin{abstract}
We study the two dual notions of prime avoidance and prime absorbance.
We generalize the classical prime avoidance lemma to radical ideals.
A number of new criteria are provided for an abstract ring to be C.P. (every set of primes satisfies avoidance) or P.Z. (every set of primes satisfies absorbance).
Special consideration is given to the interaction with chain conditions and Noetherian-like properties.
It is shown that a ring is both C.P. and P.Z. iff it has finite spectrum.
\end{abstract}

\maketitle

\section{Introduction}

The prime avoidance lemma is one of the most fundamental results in commutative algebra: if an ideal is contained in a finite union of prime ideals, then it is already contained in one of them.
The set-theoretic dual result -- referred to as \emph{prime absorbance} -- is also useful (and follows directly from the definition of primeness): if a finite intersection of ideals is contained in a prime ideal, then one of them is already contained in the prime.
However, both results fail for infinite families in general.
For example, infinite prime avoidance already fails in the ring $k[x,y]$ (cf. Example \ref{ex:avoidance}(5)), and infinite prime absorbance fails in the ring of integers $\mathbb{Z}$. 

With this in mind, the main goal of this paper is to study the dual notions of prime avoidance and prime absorbance, especially in the infinite case.
Infinite prime avoidance has been periodically investigated over the years, see e.g. \cite{Chen}, \cite{Pakala-Shores}, \cite{Ries-Vis}, \cite{Sharp-Vamos}, and \cite{Smith}.
Dually, infinite prime absorbance has been studied in \cite[\S V]{Picavet} and \cite[\S4]{Tarizadeh}.
In \cite{Karamzadeh}, the prime avoidance lemma is also proven for non-commutative rings. 

Contrary to what one might initially expect, avoidance is not strictly limited to prime ideals.
In Section 2, we formulate the avoidance property in general and show that it passes to intersections in a certain specific sense, cf. Theorem \ref{thm:intersectionAvoidance}.
This allows us to generalize the classical prime avoidance lemma to radical ideals, cf. Theorem \ref{thm:radical_avoidance} and Corollary \ref{Corollary E.D. radical avoid}. 

Section 3 investigates the rings in which every set of primes has the avoidance property, the so-called compactly packed (or C.P.) rings.
Dually, Section 4 investigates the rings in which every set of primes has the absorbance property, which we name a properly zipped (or P.Z.) ring. Although they have received less attention in the literature, P.Z. rings admit a number of interesting and natural characterizations, e.g. a ring is P.Z. iff any union of Zariski-closed sets is Zariski-closed. 

A recurring theme is the interplay of the C.P. and P.Z. properties with chain conditions and Noetherian-like properties.
For instance, it is shown that P.Z. rings are semilocal, and satisfy d.c.c. on both prime ideals and finitely generated radical ideals.
Theorem \ref{Thm c.p.+P.Z.=finite primes} characterizes the rings which are both C.P. and P.Z. as the rings with only finitely many prime ideals.
The flat topology on spectra of C.P. and P.Z. rings is also investigated, see Propositions \ref{Prop flat c.p.} and \ref{prop:flat P.Z.}.
The P.Z. rings of dimension $1$ are characterized in Theorem \ref{thm:1dimPZ}.
Finally, Section 5 concludes with various examples. 

In this paper, all rings are commutative with $1 \ne 0$.
The nilradical is denoted by $\mathfrak{N}$.
For a ring $R$, there is a (unique) topology, called the flat topology, on $\Spec(R)$ for which the collection of $V(I)$, where $I$ is a finitely generated ideal of $R$, forms a base of open sets.
If $\mathfrak{p}$ is a prime ideal of $R$, then $\Lambda(\mathfrak{p}) := \{\mathfrak{q} \in \Spec(R) : \mathfrak{q} \subseteq \mathfrak{p} \}$ is the flat closure of the point $\mathfrak{p} \in \Spec(R)$.
For more information see e.g. \cite{Tarizadeh}.

\section{General avoidance}

Let $R$ be a ring, and $\mathcal{S}$ a set of ideals of $R$.
We say that $\mathcal{S}$ \emph{satisfies avoidance} if for any ideal $J$ of $R$, whenever $\displaystyle J \subseteq\bigcup_{I \in S} I$, then $J \subseteq I$ for some $I \in S$. 

\begin{example} \label{ex:avoidance}
We illustrate the avoidance property with some basic examples: 
\begin{enumerate}[label=(\arabic*), leftmargin=*]
\item Any set of $\le 2$ ideals satisfies avoidance: if an ideal is contained in a union of $2$ ideals, then it is contained in one of them. 
\item Any finite set of prime ideals satisfies avoidance: this is the classical prime avoidance lemma. 
\item If $R$ contains an infinite field $k$, then any finite set of ideals satisfies avoidance: no $k$-vector space is a finite union of proper subspaces. 
\item The set of maximal ideals $\Max(R)$ satisfies avoidance: if an ideal consists of nonunits, then it is contained in a maximal ideal. 
\item The avoidance property need not pass to subsets or supersets.
For instance a maximal ideal may be contained in the union of the other maximal ideals (e.g. $(x,y)$ in  $k[x,y]$). 
\end{enumerate}
\end{example}

In the following result, the prime avoidance lemma is generalized for radical ideals. 

\begin{theorem} $($Radical avoidance$)$ \label{thm:radical_avoidance} If an ideal $I$ of a ring $R$ is contained in the union of a finite family $\{ I_k \}$ of radical ideals of $R$, then $I\subseteq I_{k}$ for some $k$. 
\end{theorem}

\begin{proof} Suppose $I$ is not contained in any of the $I_{k}$, so there exists $f_k \in I \setminus I_k$ for each $k$.
Then $I_k$ does not meet the multiplicative set $\{1, f_k, f_k^2, \ldots \}$, so there is a prime ideal $\mathfrak{p}_k$ of $R$ containing $I_k$ that also does not contain $f_k$.
Then clearly $I \subseteq \bigcup \limits_{k} \mathfrak{p}_{k}$, but this is in contradiction with the prime avoidance lemma.
\end{proof}

\begin{remark}
In Theorem \ref{thm:radical_avoidance}, just like in the usual prime avoidance lemma \cite[Theorem 3.61]{Sharp}, we may assume that two of the $I_{k}$'s are arbitrary ideals (not necessarily radical). 
\end{remark}

Similarly, the version of prime avoidance given by Edward Davis (cf. e.g. \cite[Ex. 16.8]{Matsumura} or \cite[Theorem 3.64]{Sharp}) can also be generalized to radical ideals. 

\begin{corollary} $($Davis' radical avoidance$)$ \label{Corollary E.D. radical avoid} Let $\{ I_k \}$ be a finite family of radical ideals of a ring $R$ and $f\in R$.
If $I$ is an ideal of $R$ such that $Rf+I\nsubseteq\bigcup\limits_{k}I_{k}$, then there exists $g\in I$ such that $f+g\notin\bigcup\limits_{k}I_{k}$. 
\end{corollary}

\begin{proof} For each $k$, there exists a prime ideal $\mathfrak{p}_{k}$ of $R$ such that $I_{k}\subseteq\mathfrak{p}_{k}$ but $Rf+I\nsubseteq\mathfrak{p}_{k}$.
So by prime avoidance, $Rf+I\nsubseteq\bigcup\limits_{k}\mathfrak{p}_{k}$.
By Davis' prime avoidance, there exists $g\in I$ such that $f+g\notin\mathfrak{p}_{k}$ for all $k$. \end{proof}

The above results exemplify the fact that some properties of prime ideals can be generalized to radical ideals.
In fact, what the proof of Theorem \ref{thm:radical_avoidance} shows is that the avoidance property passes to intersections in the following sense: 

\begin{theorem} \label{thm:intersectionAvoidance}
Let $\mathcal{S}$ be a set of ideals of $R$, and let $\mathcal{T}$ be the set of all intersections of ideals in $\mathcal{S}$.
If every subset of $\mathcal{S}$ satisfies avoidance, then every subset of $\mathcal{T}$ satisfies avoidance as well. 
\end{theorem}

\begin{proof}
Let $\{I_t \mid t \in T\}$ be a subset of $\mathcal{T}$, and suppose $\displaystyle J \subseteq \bigcup_{t \in T} I_t$, but $J \not \subseteq I_t$ for all $t \in T$.
For every $t \in T$, choose $f_t \in J \setminus I_t$.
Since $I_t$ is an intersection of elements of $\mathcal{S}$, there exists $K_t \in \mathcal{S}$ such that $I_t \subseteq K_t$ and $f_t \not \in K_t$.
Then $\displaystyle J \subseteq \bigcup_{t \in T} I_t \subseteq \bigcup_{t \in T} K_t$, but $J \not \subseteq K_t$ for all $t$, so $\{K_t \mid t \in T\} \subseteq \mathcal{S}$ does not satisfy avoidance, a contradiction. \end{proof}

Note that if $\mathcal{S} = \Spec R$, then the proof of Theorem \ref{thm:intersectionAvoidance}, with classical prime avoidance, yields an alternative proof of Theorem \ref{thm:radical_avoidance}.

\section{Prime avoidance and C.P. rings}

We now investigate the rings in which every set of primes satisfies avoidance -- these are the so-called \emph{compactly packed} (or C.P.) rings.
That is, $R$ is C.P. if whenever an ideal $I$ is contained in the union of a family $\{ \mathfrak{p}_i \}$ of prime ideals, then $I\subseteq\mathfrak{p}_i$ for some $i$.
C.P. rings have been studied in the literature, see e.g. \cite{Pakala-Shores}, \cite{Ries-Vis} and \cite{Smith}. 

We first record various ring-theoretic constructions which preserve the C.P. property:

\begin{proposition} \label{prop:CPconstructions}
Let $R$ be a ring.
\begin{enumerate}[label=(\arabic*)]
    \item If $R/\mathfrak{N}$ is C.P., then $R$ is C.P.
    \item If $R$ is C.P., then so is any quotient or localization of $R$.
    \item A finite product of C.P. rings is C.P.
\end{enumerate}
\end{proposition}

\begin{proof} We illustrate the proof of (1). Assume $R/\mathfrak{N}$ is C.P. If $\{ \mathfrak{p}_i \}$ are primes of $R$ and $I$ is an ideal of $R$ with $I \subseteq \bigcup\limits_{i} \mathfrak{p}_{i}$,
then $(I+\mathfrak{N})/\mathfrak{N} \subseteq \bigcup
\limits_{i}\mathfrak{p}_{i}/\mathfrak{N}$. Then
$(I+\mathfrak{N})/\mathfrak{N} \subseteq \mathfrak{p}_{i}/\mathfrak{N}$ for some $i$, so $I \subseteq \mathfrak{p}_{i}$.

The proofs of (2) and (3) are straightforward and left as exercises. \end{proof}

We now turn towards various characterizations of the C.P. property.
As we will see, the C.P. property turns out to imply Noetherianness of the prime spectrum in the Zariski topology, so we first recall some criteria for this to occur.

\begin{proposition} \label{prop:noetherianZariskiTop} Consider the following conditions on a ring $R$: 
\begin{enumerate}
\item $\Spec(R)$ is a Noetherian space in the Zariski topology 
\item $R$ satisfies the ascending chain condition on radical ideals 
\item Every radical ideal of $R$ is the radical of a finitely generated ideal 
\item Every ideal of $R$ has only finitely many minimal primes. 
\end{enumerate}
Then $\mathbf{(i)} \Leftrightarrow \mathbf{(ii)} \Leftrightarrow \mathbf{(iii)} \Rightarrow \mathbf{(iv)}$, and $\mathbf{(iv)} \Rightarrow \mathbf{(i)}$ if $\dim R$ is finite.
\end{proposition}

\begin{proof} Cf. \cite[Propositions 1.1 and 2.1]{Ohm}. \end{proof}

We are now ready to state our characterizations of C.P. rings.
The following result improves on \cite[Theorem 1]{Pakala-Shores}, \cite[Theorem 1.1]{Ries-Vis} and \cite{Smith}
with the addition of conditions $\mathbf{(v)}$ and $\mathbf{(vi)}$. 

\begin{theorem}\label{Theorem iv 2020} The following are equivalent for a ring $R$: 
\begin{enumerate}
\item $R$ is a C.P. ring. 
\item If a prime ideal $\mathfrak{p}$ of $R$ is contained in the union of a family $\{ \mathfrak{p}_i \}$ of prime ideals of $R$, then $\mathfrak{p}\subseteq\mathfrak{p}_{i}$ for some $i$. 
\item Every radical ideal of $R$ is the radical of a principal ideal. 
\item Every prime ideal of $R$ is the radical of a principal ideal. 
\item $\Spec(R)$ is a Noetherian space in the Zariski topology, and for any $f,g\in R$ there exists $h\in R$ such that $V(f)\cap V(g)=V(h)$. 
\item If an ideal $I$ of $R$ is contained in the union of a family $\{ I_k \}$ of radical ideals of $R$, then $I\subseteq I_{k}$ for some $k$. 
\end{enumerate}
\end{theorem}

\begin{proof} $\mathbf{(i)}\Leftrightarrow\mathbf{(ii)}$ is an easy exercise, see e.g. \cite{Smith}.

For $\mathbf{(i)}\Rightarrow\mathbf{(iii)} \Rightarrow\mathbf{(iv)}\Rightarrow\mathbf{(i)}$ see
\cite{Smith} or \cite[Theorem 1]{Pakala-Shores}. 

$\mathbf{(iii)}\Rightarrow\mathbf{(v)}:$ By Proposition \ref{prop:noetherianZariskiTop}$\mathbf{(iii)}$, $\Spec(R)$ is Noetherian.
For any $f,g \in R$, by hypothesis there exists $h\in R$ such that $\sqrt{(f,g)}=\sqrt{(h)}$.
It follows that $V(f)\cap V(g) = V(f,g) = V(h)$. 

$\mathbf{(v)}\Rightarrow\mathbf{(iii)}:$ Let $I$ be a radical ideal of $R$.
By Proposition \ref{prop:noetherianZariskiTop}$\mathbf{(iii)}$, there exists a finitely generated ideal $J$ of $R$ such that $V(I)=V(J)$.
By hypothesis, there exists $h\in R$ such that $V(J)=V(h)$.
It follows that $I=\sqrt{(h)}$. 

$\mathbf{(ii)}\Leftrightarrow\mathbf{(vi)}:$ This follows immediately from Theorem \ref{thm:intersectionAvoidance}. \end{proof}

\begin{remark} \label{rem:CPimpliesNoetherianSpectrum}
\begin{enumerate}[label=(\arabic*), leftmargin=*]
\item It follows from Proposition \ref{prop:noetherianZariskiTop} and Theorem \ref{Theorem iv 2020} that every C.P. ring satisfies the ascending chain condition on radical ideals and has finitely many minimal primes. 
\item Recall that the arithmetic rank of an ideal $I$ is the least number of elements required to generate $I$ up to radical, i.e.
\[
\ara I := \inf \big\{ n \mid \exists a_1, \ldots, a_n \in R \text{ with } \sqrt{(a_1, \ldots, a_n)} = \sqrt{I} \big\}
\]
Another way to phrase the proof of Theorem \ref{Theorem iv 2020} is: $\mathbf{(v)}$ says exactly that $\ara I < \infty$ for all ideals $I$ and if $\ara I < \infty$, then $\ara I \le 1$.
This is clearly equivalent to $\ara I \le 1$ for all ideals $I$, which is $\mathbf{(iii)}$. 
\item Since $\height I \le \ara I$ for all ideals in a Noetherian ring, it follows from Theorem \ref{Theorem iv 2020}$\mathbf{(iv)}$ that a Noetherian C.P. ring has dimension $\le 1$. 
\end{enumerate}
\end{remark}

\begin{proposition}\label{Prop flat c.p.} Let $R$ be a C.P. ring.
\begin{enumerate}[label=(\arabic*)]
    \item The collection of $V(f)$ with $f\in R$ forms a base for the flat opens of $\Spec(R)$.
    \item The flat closed subsets of $\Spec(R)$ are precisely of the form $\Ima\pi^{\ast}$ where $\pi:R\rightarrow S^{-1}R$ is the canonical map and $S$ is a multiplicative subset of $R$.
\end{enumerate}
\end{proposition}

\begin{proof} (1): If $I$ is a finitely generated ideal of $R$, then by Theorem \ref{Theorem iv 2020}$\mathbf{(iii)}$ there exists $f\in R$ such that $\sqrt{I} = \sqrt{(f)}$, so $V(I) = V(f)$. 

(2): Clearly every subset of the given form is flat closed.
Conversely, if $E \subseteq \Spec R$ is a flat closed then $E\subseteq\Ima\pi^{\ast}$ where $S := R \setminus \bigcup\limits_{\mathfrak{p}\in E}\mathfrak{p}$ and $\pi:R\rightarrow S^{-1}R$ is the canonical map.
If $\mathfrak{q} \in \Ima\pi^{\ast}$ then $\mathfrak{q}\subseteq\bigcup\limits_{\mathfrak{p}\in E}\mathfrak{p}$.
Since $R$ is C.P., this implies $\mathfrak{q}\subseteq\mathfrak{p}$ for some $\mathfrak{p}\in E$.
It follows that $\mathfrak{q}\in E$, since each flat closed subset is stable under generalization.
Therefore $E = \Ima\pi^{\ast}$. \end{proof}

\section{Prime absorbance and P.Z. rings}

The dual notion of a C.P. ring can be defined as follows.
We say that a ring $R$ is a \emph{properly zipped} (or P.Z.) ring if whenever a prime ideal $\mathfrak{p}$ of $R$ contains the intersection of a family $\{ \mathfrak{p}_i \}$ of prime ideals of $R$, then $\mathfrak{p}_{i}\subseteq\mathfrak{p}$ for some $i$. 

We first give the analogue of Proposition \ref{prop:CPconstructions}, whose proof we leave as an exercise:

\begin{proposition} \label{prop:PZconstructions}
Let $R$ be a ring.
\begin{enumerate}[label=(\arabic*)]
    \item If $R/\mathfrak{N}$ is P.Z., then $R$ is P.Z.
    \item If $R$ is P.Z., then so is any quotient or localization of $R$.
    \item A finite product of P.Z. rings is P.Z.
\end{enumerate}
\end{proposition}

Next, we turn towards a characterization of P.Z. rings.
The following result simplifies and improves on \cite[Theorem 4.2]{Tarizadeh}, with the addition of condition $\mathbf{(vii)}$.

\begin{theorem}\label{Theorem II} The following are equivalent for a ring $R$: 
\begin{enumerate}
\item $R$ is a P.Z. ring. 
\item If $\mathfrak{p}$ is a prime ideal of $R$, then there exists $f\in R$ such that $\Lambda(\mathfrak{p})=D(f)$. 
\item If $\mathfrak{p}$ is a prime ideal of $R$, then there exists $f\in R\setminus\mathfrak{p}$ such that the canonical map $R_{f}\rightarrow R_{\mathfrak{p}}$ is an isomorphism. 
\item If $\mathfrak{p}$ is a prime ideal of $R$, then the localization map $R\rightarrow R_{\mathfrak{p}}$ is of finite presentation. 
\item The Zariski opens of $\Spec(R)$ are stable under arbitrary intersections. 
\item $\Spec(R)$ is a Noetherian space in the flat topology. 
\item If a prime ideal $\mathfrak{p}$ of $R$ contains the intersection of a family $\{ I_k \}$ of radical ideals of $R$, then $I_{k}\subseteq\mathfrak{p}$ for some $k$. 
\end{enumerate}
\end{theorem}

\begin{proof} $\mathbf{(i)}\Rightarrow\mathbf{(ii)}:$ If $X := \Spec(R) \setminus \Lambda(\mathfrak{p})$, then $I := \bigcap\limits_{\mathfrak{q}\in X} \mathfrak{q}$ is not contained in $\mathfrak{p}$.
Thus there exists $f\in I \setminus \mathfrak{p}$.
It follows that $\Lambda(\mathfrak{p}) = D(f)$. 

$\mathbf{(ii)}\Rightarrow\mathbf{(iii)}:$ If $g\in R\setminus\mathfrak{p}$ then $g/1$ is invertible in $R_{f}$, since $\Lambda(\mathfrak{p}) = D(f)$.
Thus by the universal property of localization, there exists a (unique) ring map $\phi:R_{\mathfrak{p}}\rightarrow R_{f}$ such that $\pi_{f}=\phi\circ\pi_{\mathfrak{p}}$ where $\pi_{f}:R_{f}\rightarrow R_{\mathfrak{p}}$ and $\pi_{\mathfrak{p}}:R\rightarrow R_{\mathfrak{p}}$ are the canonical maps.
It follows that $\phi\circ\pi_{f}$ and $\pi_{f}\circ\phi$ are the identity maps. 

$\mathbf{(iii)}\Rightarrow\mathbf{(iv)}:$ The isomorphism $R_{f} \cong R[X]/(fX-1)$ gives that $R\rightarrow R_{\mathfrak{p}}$ is of finite presentation. 

$\mathbf{(iv)}\Rightarrow\mathbf{(ii)}:$ It is well known that every flat ring map which is also of finite presentation induces a Zariski open map on prime spectra.
In particular, $\Lambda(\mathfrak{p})$ is a Zariski open of $\Spec(R)$.
Thus we may write $\Lambda(\mathfrak{p})=\bigcup\limits_{i}D(f_{i})$.
Therefore $\mathfrak{p}\in D(f_{k})$ for some $k$, and so $\Lambda(\mathfrak{p})=D(f_{k})$. 

$\mathbf{(ii)}\Rightarrow\mathbf{(v)}:$ It suffices to prove the assertion for basic Zariski opens.
If $\mathfrak{p}\in\bigcap\limits_{i}D(f_{i})$ then there exists $f\in R$ such that $\Lambda(\mathfrak{p})=D(f)$.
It follows that $\mathfrak{p}\in D(f)\subseteq\bigcap\limits_{i}D(f_{i})$. 

$\mathbf{(v)}\Rightarrow\mathbf{(vi)}:$ It suffices to show that every flat open $U$ of $\Spec(R)$ is quasi-compact.
If $U=\bigcup\limits_{k}V(I_{k})$ where each $I_{k}$ is a finitely generated ideal of $R$, then by hypothesis, $U=V(I)$ for some ideal $I$ of $R$.
But for any ring $R$, $V(I)$ is a quasi-compact subset of $\Spec(R)$ in the flat topology.

$\mathbf{(vi)}\Rightarrow\mathbf{(ii)}:$ Since $\Lambda(\mathfrak{p})$ is a flat closed subset of $\Spec(R)$, there exists a finitely generated ideal $I$ of $R$ such that $\Spec(R) \setminus \Lambda(\mathfrak{p}) = V(I)$, because every subspace of a Noetherian space is quasi-compact.
Thus $\Lambda(\mathfrak{p})=D(f)$ for some $f\in I$. 

$\mathbf{(ii)}\Rightarrow\mathbf{(vii)}:$ Choose $f\in R$ such that $\Lambda(\mathfrak{p})=D(f)$.
Then $f\notin I_{k}$ for some $k$, so $I_{k}\cap S=\emptyset$ where $S=\{1,f,f^{2},\ldots\}$. Hence there exists a prime ideal $\mathfrak{q}$ of $R$ such that $I_{k}\subseteq\mathfrak{q}$ and $f\notin\mathfrak{q}$. Thus $I_{k}\subseteq\mathfrak{q}\subseteq\mathfrak{p}$.  

$\mathbf{(vii)}\Rightarrow\mathbf{(i)}:$ Clear. \end{proof}

The next two corollaries codify important properties of P.Z. rings, and serve as a dual to Remark \ref{rem:CPimpliesNoetherianSpectrum}(1). 

\begin{corollary} Every P.Z. ring satisfies the descending chain condition on both prime ideals and finitely generated radical ideals. 
\end{corollary}

\begin{proof} Let $R$ be a P.Z. ring. If $\mathfrak{p}_{1}\supseteq\mathfrak{p}_{2}\supseteq\ldots$ is a descending chain of prime ideals of $R$, then $\Lambda(\mathfrak{p}_{1})\supseteq\Lambda(\mathfrak{p}_{2})
\supseteq\ldots$ is a descending chain of flat closed subsets of $\Spec(R)$.
By Theorem \ref{Theorem II}$\mathbf{(vi)}$, there is some $n$ such that $\mathfrak{p}_{n}=\mathfrak{p}_{i}$ for all $i\geqslant n$.

Similarly, if $I_{1}\supseteq I_{2}\supseteq\ldots$ is a descending chain of finitely generated radical ideals of $R$, then $V(I_{1})\subseteq V(I_{2})\subseteq\ldots$ is an ascending chain of flat open subsets of $\Spec(R)$.
Again by Theorem \ref{Theorem II}$\mathbf{(vi)}$, there is some $n$ such that $I_{n}=I_{k}$ for all $k\geqslant n$. \end{proof}

\begin{corollary}\label{Corollary VI max finite} Let $R$ be a P.Z. ring.
Then $\Max(R)$ is a finite set. 
\end{corollary}

\begin{proof} By Theorem \ref{Theorem II}$\mathbf{(ii)}$, for every maximal ideal $\mathfrak{m}$ of $R$, there exists $x_{\mathfrak{m}}\in R$ such that $\Lambda(\mathfrak{m})=D(x_{\mathfrak{m}})$.
Since $\Max(R)$ is always Zariski quasi-compact, there is a finite subcover $\Max(R)\subseteq\bigcup\limits_{i=1}^{d}D(x_{i})$ where $x_{i}:=x_{\mathfrak{m}_{i}}$ for $i = 1, \ldots, d$.
Then $\Max(R)=\{\mathfrak{m}_{1},\ldots,\mathfrak{m}_{d}\}$. \end{proof}

\begin{proposition}\label{prop:flat P.Z.} For a P.Z. ring $R$, the flat opens of $\Spec(R)$ are precisely sets of the form $V(I)$ with $I$ a finitely generated ideal of $R$. 
\end{proposition}

\begin{proof} Clearly every subset of the given form is flat open.
Conversely, let $U$ be a flat open subset of $\Spec(R)$. Then by Theorem \ref{Theorem II}$\mathbf{(vi)}$, $U$ is quasi-compact in the flat topology.
Thus there exists a finitely generated ideal $I$ of $R$ such that $U = V(I)$. \end{proof}

We next characterize the zero-dimensional C.P. and P.Z. rings:

\begin{theorem} \label{prop:zerodimCP} The following are equivalent for a ring $R$: 
\begin{enumerate}
\item $R$ is a zero-dimensional C.P. ring. 
\item $R$ is a zero-dimensional P.Z. ring. 
\item $R/\mathfrak{N}$ is a finite product of fields. 
\end{enumerate}
\end{theorem}

\begin{proof} Suppose $\dim R = 0$. Then $\Min(R) = \Max(R) = \Spec(R)$, so if $R$ is C.P. (resp. P.Z.), then $\Spec(R)$ is finite by Remark \ref{rem:CPimpliesNoetherianSpectrum}(1) (resp. Corollary \ref{Corollary VI max finite}), say $\Spec(R)=\{\mathfrak{p}_{1}, \ldots ,\mathfrak{p}_{n}\}$. Thus by the Chinese remainder theorem, $R/\mathfrak{N} \cong R/\mathfrak{p}_1 \times \ldots \times R/\mathfrak{p}_n$ is a finite product of fields. This shows $\mathbf{(i)} \Rightarrow \mathbf{(iii)}$ (resp. $\mathbf{(ii)} \Rightarrow \mathbf{(iii)}$).

On the other hand, a field is both C.P. and P.Z., so the converse follows from Propositions \ref{prop:CPconstructions} and \ref{prop:PZconstructions}. \end{proof}

\begin{corollary}\label{Corollary II} Let $R$ be a Boolean ring.
Then $R$ is C.P. if and only if $R$ is P.Z. if and only if $R$ is finite.
In particular, for a set $X$, the power set ring $\mathcal{P}(X)$ is C.P. or P.Z. if and only if $X$ is finite. 
\end{corollary}

\begin{proof} By Theorem \ref{prop:zerodimCP} it suffices to show that a Boolean P.Z. ring is finite.
This follows from the proof of Theorem \ref{prop:zerodimCP}$\mathbf{(ii)}\Rightarrow\mathbf{(iii)}$, since every residue field of a Boolean ring is $\mathbb{Z}/2\mathbb{Z}$. \end{proof}

We can now characterize the rings which satisfy both avoidance and absorbance.
The proof relies significantly on topological methods.

\begin{theorem}\label{Thm c.p.+P.Z.=finite primes} The following are equivalent for a ring $R$: 
\begin{enumerate}
\item $R$ is both C.P. and P.Z. 
\item $\Spec(R)$ is Noetherian in both the Zariski and flat topologies. 
\item $\Spec(R)$ is a finite set.
\end{enumerate}
\end{theorem}

\begin{proof} $\mathbf{(i)}\Rightarrow\mathbf{(ii)}:$ See Theorem \ref{Theorem iv 2020}$\mathbf{(v)}$ and Theorem \ref{Theorem II}$\mathbf{(vi)}$.

$\mathbf{(ii)}\Rightarrow\mathbf{(iii)}:$ See \cite[Theorem 4.5]{Tarizadeh}.

$\mathbf{(iii)}\Rightarrow\mathbf{(i)}:$ Clear. \end{proof}

Returning to P.Z. rings, the following proposition states that if the ring itself is Noetherian (not just the spectrum), then it suffices to check only minimal primes in Theorem \ref{Theorem II}$\mathbf{(ii)}$, i.e. a Noetherian ring is P.Z. iff every minimal prime is an isolated point in the Zariski topology.
This is a dual result to \cite[Theorem 1]{Pakala-Shores} which may be phrased as: a Noetherian ring is C.P. iff every maximal ideal is an isolated point in the flat topology. 

\begin{theorem}\label{Theorem III P.Z. Noeth.} Let $R$ be a Noetherian ring.
Then $R$ is P.Z. if and only if for every minimal prime $\mathfrak{p}$ of $R$, there exists $f\in R$ such that $D(f)=\{\mathfrak{p}\}$. 
\end{theorem}

\begin{proof} The direction $\Rightarrow$ follows from Theorem \ref{Theorem II}$\mathbf{(ii)}$.
Conversely, it suffices to show that $V(\mathfrak{p})$ is finite for all $\mathfrak{p}\in\Min(R)$ (so that $\Spec(R)$ is finite).
Pick $\mathfrak{p} \in \Min(R)$.
If $V(\mathfrak{p})$ is infinite then it contains infinitely many height one prime ideals, all of which contain $f$.
Then the image of $f$ in $R/\mathfrak{p}$ generates a height one ideal with infinitely many minimal primes, which is impossible since $R/\mathfrak{p}$ is Noetherian. \end{proof}

We next rephrase the P.Z. property for rings of dimension 1: 

\begin{theorem} \label{thm:1dimPZ}
Let $R$ be a 1-dimensional ring. Then $R$ is P.Z. iff

\begin{enumerate}[label=(\arabic*)]
    \item $R$ is semilocal,
    \item For all $\mathfrak{p} \in \Min(R)$, $\displaystyle \bigcap_{\mathfrak{q} \in \Min(R) \setminus \{\mathfrak{p}\}} \mathfrak{q} \ne \mathfrak{N}$, and
    \item For all $\mathfrak{m} \in \Max(R)$, $\displaystyle \mathfrak{m} + \bigcap_{\mathfrak{q} \in \Min(R) \setminus \Lambda(\mathfrak{m})} \mathfrak{q} = R$. 
\end{enumerate}
\end{theorem}

\begin{proof} Any P.Z. ring satisfies conditions (2) and (3), as well as (1) by Corollary \ref{Corollary VI max finite}.
Conversely, we show that conditions (1)-(3) are equivalent to Theorem \ref{Theorem II}$\mathbf{(ii)}$.
If $\mathfrak{m} \in \Max(R)$, then by (3) we may choose $\displaystyle f_\mathfrak{m} \in \Big( \bigcap_{\mathfrak{q} \in \Min(R) \setminus \Lambda(\mathfrak{m})} \mathfrak{q} \Big) \setminus \mathfrak{m}$.
Then choosing $\displaystyle g_\mathfrak{m} \in \Big( \bigcap_{\mathfrak{m}' \in \Max(R) \setminus \{\mathfrak{m}\}} \mathfrak{m}' \Big) \setminus \mathfrak{m}$ (which is possible since the intersection is finite) gives $\Lambda(\mathfrak{m}) = D(f_\mathfrak{m} g_\mathfrak{m})$.
If now $\mathfrak{p}$ is a non-maximal (hence minimal) prime, then (2) implies that there exists $\displaystyle x_{\mathfrak{p}} \in \Big( \bigcap_{\mathfrak{q} \in \Min(R) \setminus \{\mathfrak{p}\}} \mathfrak{q} \Big) \setminus \mathfrak{p}$.
As before, if $\displaystyle y_\mathfrak{p} \in \Big( \bigcap_{\mathfrak{m} \in V(\mathfrak{p}) \cap \Max(R)} \mathfrak{m} \Big) \setminus \mathfrak{p}$, then $\Lambda(\mathfrak{p}) = \{\mathfrak{p}\} = D(x_\mathfrak{p} y_\mathfrak{p})$. \end{proof}

\begin{remark}
Any ring with finitely many minimal primes satisfies conditions (2) and (3) in Theorem \ref{thm:1dimPZ}. 
\end{remark}

\begin{corollary} \label{cor:1dimPZ}
Let $R$ be a 1-dimensional reduced local ring.
Then $R$ is P.Z. iff $\displaystyle \bigcap_{\mathfrak{q} \in \Min(R) \setminus\{\mathfrak{p}\}} \mathfrak{q} \ne 0$ for all $\mathfrak{p} \in \Min(R)$. 
\end{corollary}

\begin{proof} This follows immediately from Theorem \ref{thm:1dimPZ}. \end{proof}

Finally, we investigate Goldman ideals in P.Z. rings.
The following result improves on \cite[Theorems 18, 24]{Kaplansky} and \cite[\S I, Prop. 1]{Picavet} with the addition of conditions $\mathbf{(vi)}$ and $\mathbf{(vii)}$.

\begin{theorem}\label{Theorem Goldman} The following are equivalent for an integral domain $R$ with field of fractions $K$: 
\begin{enumerate}
\item $K$ is a simple extension of $R$. 
\item $K$ is a finitely generated algebra over $R$. 
\item There exists a maximal ideal $\mathfrak{m}$ of $R[X]$ such that $\mathfrak{m}\cap R=0$. 
\item $\{0\}$ is a locally closed subset of $\Spec(R)$. 
\item The zero ideal of $R$ is an isolated point of $\Spec(R)$. 
\item If $\{ \mathfrak{p}_i \}$ is a family of non-zero prime ideals of $R$, then $\bigcap\limits_{i}\mathfrak{p}_{i}\neq0$. 
\item There is some non-zero $f\in R$ such that for each non-zero $g\in R$ then $f\in\sqrt{(g)}$. 
\end{enumerate}
\end{theorem}

\begin{proof} For $\mathbf{(i)}\Leftrightarrow\mathbf{(ii)}\Leftrightarrow\mathbf{(iii)}$ see \cite[Theorems 18, 24]{Kaplansky}. For $\mathbf{(iii)}\Leftrightarrow\mathbf{(iv)}
\Leftrightarrow\mathbf{(v)}$ see \cite[\S I, Prop. 1]{Picavet}.

$\mathbf{(v)}\Rightarrow\mathbf{(vi)}:$ It suffices to show that the intersection of all non-zero primes is non-zero.
By hypothesis, there is some non-zero $f\in R$ such that $\{0\}=D(f)$.
Thus $f$ belongs to the above intersection.

$\mathbf{(vi)}\Rightarrow\mathbf{(vii)}:$ Choose some non-zero $f$ from the intersection of all non-zero primes.
Then $f\in\sqrt{(g)}$ for all non-zero $g$.

$\mathbf{(vii)}\Rightarrow\mathbf{(i)}:$ If $h/g\in K$ then $g$ is non-zero.
Thus there exists a natural number $n\geqslant1$ such that $f^{n}=rg$ for some non-zero $r\in R$.
So $h/g=rh/f^{n}$, which yields $K=R[1/f]$. \end{proof}

An integral domain which satisfies one of the equivalent conditions of Theorem \ref{Theorem Goldman} is called a Goldman domain.
A prime ideal $\mathfrak{p}$ of a ring $R$ is called a Goldman ideal (or $G$-ideal) of $R$ if $R/\mathfrak{p}$ is a Goldman domain.
The set of Goldman ideals of a ring $R$ is denoted by $\Gold(R)$.
For more information on this set see \cite{Picavet}. 

\begin{corollary} If a ring $R$ is a P.Z. ring, then $\Gold(R)=\Spec(R)$. 
\end{corollary}

\begin{proof} Let $\mathfrak{p}$ be a prime ideal of $R$.
If $(\mathfrak{p}_{i})$ is a subset of $V(\mathfrak{p})\setminus\{\mathfrak{p}\}$ then $\bigcap\limits_{i}\mathfrak{p}_{i} \neq \mathfrak{p}$.
Thus by Theorem \ref{Theorem Goldman}$\mathbf{(vi)}$, $\mathfrak{p}$ is a $G$-ideal. \end{proof}

\section{Examples}

\begin{example}
If $k$ is a field, then $k[x,y,z]/(xy,xz)$ is a Noetherian ring of dimension $1$, but is neither C.P. nor P.Z.: for more details see \cite[Example 5]{Chen}. 
\end{example}

\begin{example} \label{ex:infiniteProduct}
If $\{ R_i \}$ is an infinite family of rings, then $\prod_i R_i$ always has infinitely many minimal primes and maximal ideals.
By Remark \ref{rem:CPimpliesNoetherianSpectrum}(1) and Corollary \ref{Corollary VI max finite}, $\prod_i R_i$ is never C.P. nor P.Z.
Thus Propositions \ref{prop:CPconstructions}(3) and \ref{prop:PZconstructions}(3) are sharp. 
\end{example}

\begin{example}
In Theorems \ref{Theorem iv 2020}$\mathbf{(vi)}$ and \ref{Theorem II}$\mathbf{(vii)}$, the ``radical'' assumptions are necessary.
For example, if $R$ is a DVR with a uniformizer $p$, then $R$ is P.Z., but $\bigcap\limits_{k \geq 1} (p^k) = 0$.
Similarly, if $R$ is a Dedekind domain with torsion but nonzero class group, then $R$ is C.P., but for any nonprincipal ideal $I$, one has $I \subseteq \bigcup\limits_{a \in I}(a)$, but $I \not \subseteq (a)$ for any $a \in I$. 
\end{example}

\begin{example}
Neither C.P. nor P.Z. are local properties: if $R$ is an infinite product of fields (or more generally any non-Noetherian absolutely flat ring), then $R$ is neither C.P. nor P.Z., by Example \ref{ex:infiniteProduct}.
However, $R_\mathfrak{p}$ is a field for all $\mathfrak{p} \in \Spec(R)$.

Similarly, neither the C.P. nor P.Z. properties are preserved by adjoining variables.
For example if $k$ is a field then $k[x]$ is not P.Z., since $\Max(k[x])$ is infinite.
Also $k[x]$ is C.P. (being a PID), but by Remark \ref{rem:CPimpliesNoetherianSpectrum}(3), $k[x][y] = k[x,y]$ is not C.P. 
\end{example}

We have seen that among Noetherian rings, the C.P. or P.Z. rings have dimension $\le 1$.
However, in general there are rings of any finite dimension which are both C.P. and P.Z.
For example, a finite-dimensional valuation ring has finite spectrum (since the primes are totally ordered), hence is both C.P. and P.Z.
More generally, it is well known that for any finite poset $P$, there exists a ring $R$ such that $\Spec(R)$ is order-isomorphic to $P$ as a poset, see \cite{Lewis} or \cite{Hochster}. 

The following result guarantees existence of P.Z. rings with infinite spectra. 

\begin{proposition}\label{Prop Hochster} Let $R$ be a ring such that $\Spec(R)$ is a Noetherian space in the Zariski topology.
Then there exists a P.Z. ring $R'$ such that $\Spec(R')$ is in bijection with $\Spec(R)$ and this correspondence reverses the prime orders. 
\end{proposition}

\begin{proof} There exists a ring $R'$ such that $\Spec(R')$ equipped with the flat topology is homeomorphic to $\Spec(R)$ equipped with the Zariski topology, see \cite[Theorem 6 and Proposition 8]{Hochster} or \cite[Theorem 3.20]{Tarizadeh}.
Thus $\Spec(R')$ is a Noetherian space in the flat topology, so by Theorem \ref{Theorem II}$\mathbf{(vi)}$, $R'$ is a P.Z. ring.
If the homeomorphism is denoted by $\ast:\Spec(R)\rightarrow\Spec(R')$, then for primes $\mathfrak{p}\subseteq\mathfrak{q}$ of $R$, one has $\mathfrak{q}\in V(\mathfrak{p})=\overline{\{\mathfrak{p}\}}$.
Thus $\mathfrak{q}^{\ast}\in\overline{\{\mathfrak{p}^{\ast}\}}=\Lambda(\mathfrak{p}^{\ast})$, so $\mathfrak{q}^{\ast} \subseteq \mathfrak{p}^{\ast}$. \end{proof}

\begin{example} Taking $R = \mathbb{Z}$ in Proposition \ref{Prop Hochster} yields a 1-dimensional local P.Z. ring $R'$ with infinite spectrum.
Since Corollary \ref{Corollary VI max finite} implies that a P.Z. ring with infinite spectrum has Krull dimension $> 0$, this is a minimal example of a P.Z. ring with infinite spectrum.
Quotienting by the nilradical of $R'$ then gives a P.Z. ring satisfying the conditions of Corollary \ref{cor:1dimPZ}.
\end{example}

\end{document}